\documentclass[11pt, english]{article}
\usepackage{amssymb, amsmath, graphicx}
\graphicspath{{Documents}}
\usepackage[english]{babel}
\usepackage[utf8]{inputenc}
\usepackage{hyperref}
\hypersetup{hidelinks}
\usepackage{parskip}
\usepackage{amsthm}
\usepackage{mathtools}
\usepackage{relsize}
\usepackage[backend=biber]{biblatex}
\bibliography{StarvedPolytopes}{}

\newtheorem{theorem}{Theorem}[section]
\newtheorem{proposition}[theorem]{Proposition}
\newtheorem{lemma}[theorem]{Lemma}
\newtheorem{corollary}[theorem]{Corollary}
\newtheorem{definition}[theorem]{Definition}
\newtheorem{remark}[theorem]{Remark}
\newtheorem{example}[theorem]{Example}
\newtheorem{algorithm}[theorem]{Algorithm}
\newtheorem{question}[theorem]{Question}

\DeclareMathOperator{\Relint}{RelInt}
\DeclareMathOperator{\Relbd}{RelBd}
\DeclareMathOperator{\Sym}{Sym}
\DeclareMathOperator{\rank}{rank}

\title{Hyperbolic polynomials and starved polytopes}
\author{Arne Lien\footnote{This work has been supported by European Union’s Horizon 2020 research and innovation programme under
the Marie Sklodowska-Curie Actions, grant agreement 813211 (POEMA).}}
\date{\vspace{-4ex}}

\newcommand{\Addresses}{{
  \bigskip
  \footnotesize

  \textsc{A.~Lien, Fachbereich Mathematik und Statistik, Universität Konstanz,
    78457 Konstanz, Germany}\newline\nopagebreak
  \textit{E-mail address}: \texttt{arne.lien@uni-konstanz.de}
}}

\begin{document}

\maketitle

\begin{abstract}
We study sets of univariate hyperbolic polynomials that share the same first few coefficients and show that they have a natural combinatorial description akin to that of polytopes. We define a stratification of such sets in terms of root arrangements of hyperbolic polynomials and show that any stratum is either empty, a point or of maximal dimension and in the latter case we characterise its relative interior. This is used to show that the poset of strata is a graded, atomic and coatomic lattice and to provide an algorithm for computing which root arrangements are realised in such sets of hyperbolic polynomials.
\end{abstract}

The topic of this article is the study of sets of univariate hyperbolic polynomials that share the same first few coefficients. The motivation for studying such sets stems from the article \cite{riener2012degree}, where they were used to provide a new proof of Timofte's degree and half degree principle for symmetric polynomials. Therefore the sets are deeply connected with the study of multivariate symmetric polynomials.

A natural way of describing the root arrangement of a hyperbolic polynomial is to construct its partition of multiplicities. However, by also considering in which order the roots arise we get a finer description of the polynomials root arrangement which we call its composition. We shall see that by using compositions to stratify the set of hyperbolic polynomials, we get a lattice of strata that is graded, atomic and coatomic.

To establish these combinatorial properties of the set of strata we show that the strata are connected to a type of symmetric, real algebraic set called Vandermonde varieties. This connection allows us to show that the relative interior of a stratum consists of the polynomials with the largest number of distinct roots and that a stratum is either empty, a point or of the generic dimension of a nonempty Vandermonde variety.

We begin in Section \ref{intro} by introducing the sets that we are studying and defining our stratification. We look into the following example of a set of degree $5$ hyperbolic polynomials with the same first three coefficients:
\begin{center}
\includegraphics[scale=0.665]{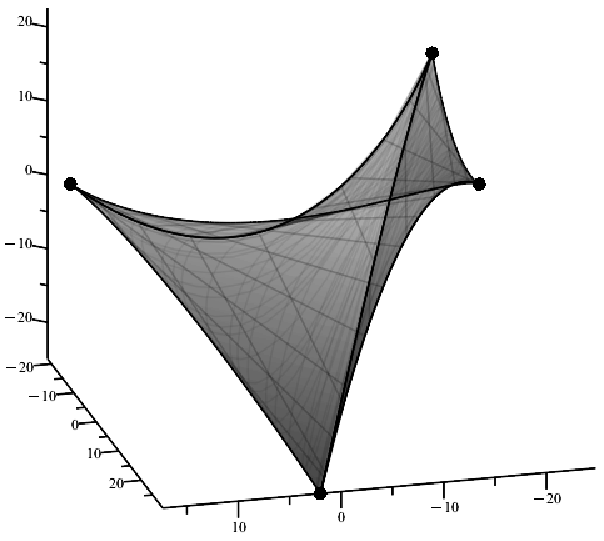}.
\end{center}
Then we show in Proposition \ref{prop:lattice} that the collection of strata, partially ordered by inclusion, is a lattice.

In Section \ref{section:vand} we show that a result from \cite{kostov1989geometric} on Vandermonde varieties can be used to describe the relative interior of the strata (Theorem \ref{thm:relint}) and to show that the strata are either empty, a point or of maximal dimension (Theorem \ref{thm:consequences}).

These results are then used in Section \ref{section:comb} to establish Theorem \ref{thm:latticeprops}, which says that the poset of strata is a graded, atomic and coatomic lattice. Finally this leads us to Algorithm \ref{alg:comps} that determines which compositions occur in our sets.

We finish with Section \ref{starve} where we discuss how a result from \cite{meguerditchian1991geometrie} on the discriminant and subdiscriminants implies that the boundary of the sets of hyperbolic polynomials have a concave-like property. In combination with Theorem \ref{thm:latticeprops}, this leads to the description "starved polytopes".

\textbf{\textit{Acknowledgements.}} I am very grateful to Claus Scheiderer for all the helpful discussions, critiques and advice along the way. I am also very grateful to Cordian Riener for bringing to my attention to some of the literature on the topic and for greatly simplifying the argument behind Theorem \ref{thm:consequences}. Lastly I would also like to thank Tobias Metzlaff whose suggestion led to the pretty picture above.

\section{Stratification}\label{intro}
\setcounter{section}{1}
\setcounter{subsection}{1}

We start off by defining the sets of univariate hyperbolic polynomials and showing how we stratify these sets. Then we look closer at the example from the introduction. We finish this section by showing that the collection of strata, partially ordered by inclusion, is a lattice.

\begin{definition}
A univariate polynomial $f\in\mathbb{R}[t]$ is called \textbf{hyperbolic} if all its roots are real.
\end{definition}

Given a hyperbolic polynomial $f$ of degree $d$, we will denote by $H_s(f)$ the set of all other hyperbolic polynomials of degree $d$ with the same $s+1$ first coefficients as $f$. That is, if $f=f_0t^d+f_1t^{d-1}+...+f_d$ is hyperbolic, then $$H_s(f):= \{h\in\mathbb{R}[t]|h \text{ is hyperbolic and }h_i =f_i \ \forall \ i \leq s\}.$$

If $a=(a_1,...,a_d)$ are the roots of $f$ then it is well known that $f_i = (-1)^ie_i(a)$, where $e_i$ is the $i^{th}$ elementary symmetric polynomial in $d$ variables. Thus the subscript $s$ refers to the number of fixed elementary symmetric polynomials.

We will let $f$ be some monic, hyperbolic polynomial of degree $d\geq 1$ throughout this article. Also, we will refer to $H_s(f)$ as a \textbf{starved polytope} although we have yet to establish this as a good description.

Often a polynomial $h=t^d+h_1t^{d-1}+...+h_d\in H_s(f)$ will be identified with the point $(h_{s+1},...,h_d)\in \mathbb{R}^{d-s}$ without specifying this change of basis. Thus when we consider topological questions, we will be equipping $H_s(f)$ with the subspace topology of the Euclidean topology on $\mathbb{R}^{d-s}$.

To stratify the starved polytopes we introduce compositions and a corresponding partial order.

\begin{definition}
A \textbf{composition} of $d$ is a tuple of positive integers, \\$u =(u_1,...,u_l)$, with $$\sum_{i=1}^lu_i=d.$$ The integers $u_i$ are called the \textbf{parts} of $u$ and $\ell(u):=l$ the \textbf{length} of $u$.
\end{definition}

We often use the shorthand $(1^d)$ for the composition whose parts are all equal to $1$. Also, when nothing further is specified, we will let $u$ denote a composition of $d$.

\begin{definition}
Let $u$ and $v$ be two compositions of a positive integer $d$. Then $v < u$ if $v$ can be obtained from $u$ by replacing some of the commas in $u$ with plus signs.
\end{definition}

To any hyperbolic polynomial $f$ of degree $d$, we associate a composition of $d$ the following way: let $a_1<a_2...<a_l$, be the distinct roots of $f$ with respective multiplicities $m_1,...,m_l$. Then the \textbf{composition of $f$} is the tuple $v(f) := (m_1,...,m_l)$, which is a composition of $d$.

The strata we will look at are given by $$H_s^u(f):=\{h\in H_s(f)|v(h)\leq u\}.$$
Note that $H^{(1^d)}_s(f) = H_s(f)$ since any composition of $d$ is smaller than $(1^d)$.

\begin{example}\label{ex:tetra}
Let $d=5$ and $s=2$ and let $$f=(t+\pi)(t+\sqrt{2})t(t-1,23456789123456789)(t-e),$$ then if we map the last three coefficients of the polynomials in $H_2(f)$, we get the three dimensional picture from the introduction.

\includegraphics[scale=0.53]{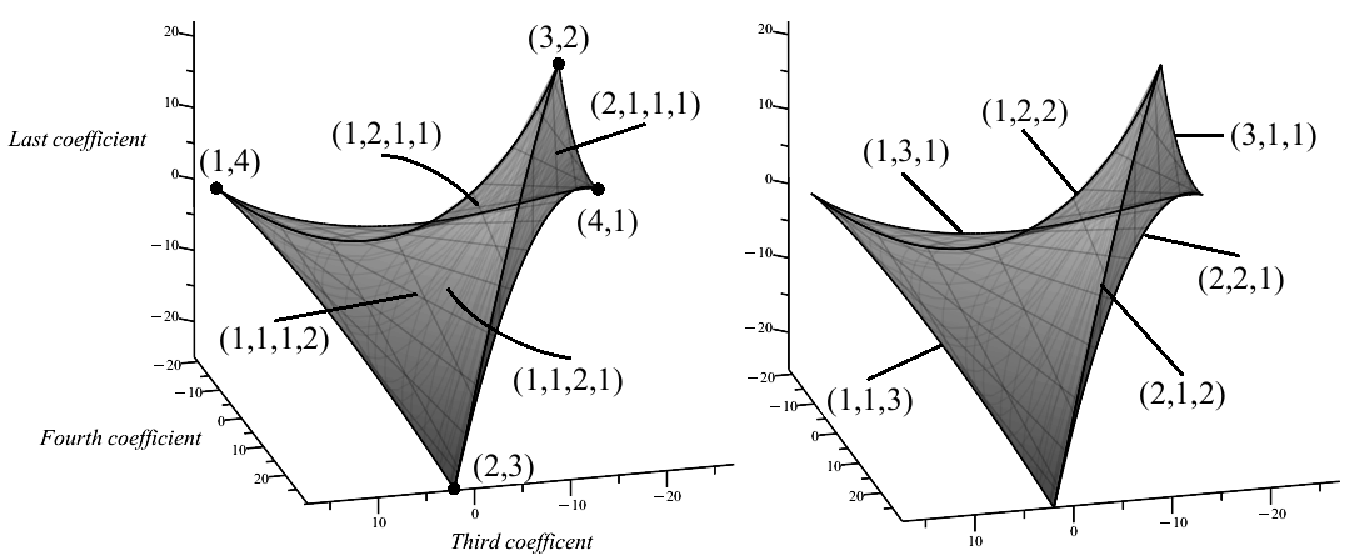}

The polynomials with no repeated roots make up the interior and the other compositions lies on the pieces of the boundary as indicated by the picture. For instance, if $u=(2,1,2)$, then the compositions $(2,3)$ and $(3,2)$ are the compositions in the picture which is smaller than $u$. Thus $H_2^u(f)$ correspond to the closest one dimensional piece of the boundary (including its endpoints).
\end{example}

\begin{remark}
The polynomial in Example \ref{ex:tetra} was chosen to give a fair representation of $H_2(f)$ for degree $5$ and its roots was chosen to illustrate that the resulting set is not dependant on the roots being particularly "nice". However, the roots where grouped relatively close together to get a picture that is suitable for an A4 page and not one that is too stretched or skewed.
\end{remark}

Next we show that the poset of strata is a lattice. If $a$ and $b$ are elements of a lattice, we denote their join by $a\vee b$ and their meet by $a\wedge b$.

\begin{lemma}\label{lem:latticecomps}
The poset of compositions of $d$ is a lattice.
\end{lemma}

\begin{proof}
We prove this by explicitly constructing the join of two compositions of d, $u$ and $v$. Having done so, the existence and uniqueness of the meet is also established since the meet of $u$ and $v$ is just the join of all compositions smaller than both $u$ and $v$.

If $a$ and $b$ are two compositions of $d$, note that $a\leq b$ if and only if $$\{a_1,a_1+a_2,...,d\}\subseteq \{b_1,b_1+b_2,...,d\}.$$ So let $M=\{u_1,u_1+u_2,...,d\}\cup \{v_1,v_1+v_2,...,d\}$ and construct the tuple $
m=(m_1,m_2,...,m_l)$ containing all the distinct elements of $M$ and where $m_i<m_{i+1}$ for any $i\in [l-1]$ (note that $m_l=d$).

Next we construct the composition $w=(m_1,m_2-m_1,m_3-m_2,...,m_l-m_{l-1})$. This is naturally a composition of $d$ and since $\{u_1,u_1+u_2,...,d\}\subseteq M$ and $\{v_1,v_1+v_2,...,d\}\subset M$, $w$ is greater than both $u$ and $v$. Also, since $M$ is by construction the unique minimal set that contains both $\{u_1,u_1+u_2,...,d\}$ and $\{v_1,v_1+v_2,...,d\}$, then $w$ is the join of $u$ and $v$.
\end{proof}

From Lemma \ref{lem:latticecomps} we immediately get that the set of strata of $H_s(f)$, partially ordered by inclusion, form a lattice. To see this let us determine the meet of two faces of $H_s(f)$, $H^u_s(f)$ and $H_s^v(f)$.

The meet of two strata must be contained in their intersection since the partial order is given by inclusion. Also, by definition the of the strata we have
$$H^u_s(f)\cap H_s^v(f)=H_s^{u\wedge v}(f),$$ so the intersection is a stratum of $H_s(f)$. Thus we have that
$$H^u_s(f)\wedge H_s^v(f)=H_s^{u\wedge v}(f)$$ and we have shown the following:

\begin{proposition}\label{prop:lattice}
The set of strata of $H_s(f)$, partially ordered by inclusion, is a lattice.
\end{proposition}

However, it is worth pointing out that just because the meet of $H^u_s(f)$ and $H_s^v(f)$ is $H_s^{u\wedge v}(f)$, this does not a priori mean that $u\wedge v$ is the only composition, $w$, such that $H^u_s(f)\wedge H_s^v(f)=H_s^w(f)$. For instance if $H_s^u(f)$ is empty, then $H^u_s(f)\wedge H_s^v(f)=H_s^u(f)$ even if $u\wedge v \neq u$. So different compositions may label the same stratum and this does indeed happen sometimes.

A similar problem can arise when we consider joins of strata. However in this case, we shall see in Section \ref{section:comb} that this can be overcome by requiring $u$ and $v$ to be the minimal compositions that can label the strata $H_s^u(f)$ and $H_s^v(f)$.

\section{Vandermonde varieties}\label{section:vand}

To describe the combinatorial structure of the lattice of strata we need to establish some geometric properties of our strata. In particular we will show that Arnold's, Givental's and Kostov's work on so called "Vandermonde varieties" (see \cite{arnol1986hyperbolic}, \cite{givental1987moments} and \cite{kostov1989geometric}) implies that $H_s^u(f)$ is either empty, a point or of dimension $\ell(u)-s$ and that in the latter case the polynomials with composition $u$ make up the relative interior.

To see the connection to their work, let the symmetric group $\Sym([d])$ act on $\mathbb{R}[x_1,...,x_d]$ by permuting the variables. Then it is well known that the elementary symmetric polynomials generate the ring of invariants (see Chapter 7.1 in \cite{cox2013ideals}). The induced action on $\mathbb{R}^d$ permutes the coordinates of the points in $\mathbb{R}^d$ and the orbit space $\mathbb{R}^d/\Sym([d])$ can be identified with the image of $\mathbb{R}^d$ under the mapping $\Pi:\mathbb{R}^d\to\mathbb{R}^d$, where $$\Pi(x)=t^d-e_1(x)t^{d-1}+...+(-1)^de_d(x)$$ and $e_i(x)$ is the $i^{th}$ elementary symmetric polynomial.

Thus the orbit space can be identified with the monic, hyperbolic polynomials of degree $d$ and by restricting $\Pi$ to the set $K_d:=\{x\in\mathbb{R}^d|x_1\leq...\leq x_d\}$, we obtain a bijection between $K_d$ and $\Pi(\mathbb{R}^d)$. So we see that the starved polytopes can be thought of as sections of the orbit space obtained by intersecting it with certain affine hyperplanes.

Similarly, if $u=(u_1,...,u_l)$ and $h\in H_s^u(f)$ has the roots $a_1\leq ... \leq a_l$, then $$h(t) = t^d-e_1(a_u)t^{d-1}+e_2(a_u)t^{d-2}...+(-1)^de_d(a_u),$$ where $a_u:=(a_1,...,a_1,...,a_l,...,a_l)$ and $a_i$ is repeated $u_i$ times. So if we define the real algebraic set $$V^u_s(f):=\{x\in\mathbb{R}^l|e_i(x_u)=(-1)^if_i \ \forall \ i\in [s]\},$$ then we see that $H_s^u(f)$ is the image of $V^u_s(f)\cap K_l$ under the mapping
$$\Pi_u(x):= t^d-e_1(x_u)t^{d-1}+...+(-1)^de_d(x_u).$$

Another set of generators for the invariant ring $\mathbb{R}[x_1,...,x_d]^{\Sym([d])}$ is the power sums $p_i(x):=\sum_{j\in [d]}x_j^i$, where $i\in [d]$. And since the first $s$ power sums can be expressed as polynomials in the first $s$ elementary symmetric polynomials using Newton's identities, we can equivalently define $V_s^u(f)$ as $\{x\in\mathbb{R}^l|p_i(x_u)=c_i \ \forall \ i\in [s]\}$, where $c_1,...,c_s\in\mathbb{R}$ are obtained from $-f_1,...,(-1)^sf_s$ using Newton's identities. Such sets are referred to as Vandermonde varieties since the Jacobian of the first $s$ power sums is a constant multiple of a Vandermonde determinant.

To make use of the previous work on Vandermonde varieties, we first need to establish that $\Pi_u$ is a homeomorphism. Note that, as with $H_s(f)$, we will be equipping the sets $V_s^u(f), K_l$ and $H_s^u(f)$ with the subspace topology of the Euclidean topology.

Before we proceed, let $B_\epsilon(a)$ denote the real open ball about $a\in \mathbb{R}^k$, of radius $\epsilon > 0$, and let $\overline{B_\epsilon(a)}$ denote its closure. Similarly, let $D_\epsilon(z)$ denote the complex open ball about $z\in \mathbb{C}^k$, of radius $\epsilon$, and let $\overline{D_\epsilon(z)}$ denote its closure.

\begin{lemma}\label{lem:closedViete}
Let $\ell(u)=l$, then $H_s^u(f)$ is closed in $\mathbb{R}^{d-s}$ and $$\Pi_u:V^u_s(f)\cap K_l\to H_s^u(f)$$ is a homeomorphism.
\end{lemma}

\begin{proof}
Note that $\Pi_u$ is a bijection and a polynomial mapping, thus it is a continuous bijection. To see that the inverse map is continuous and that $H_s^u(f)$ is closed in $\mathbb{R}^{d-s}$, we show that the image of closed sets in $V^u_s(f)\cap K_l$ are closed in $\mathbb{R}^{d-s}$.

Let $S$ be a closed subset of $V^u_s(f)\cap K_l$, then since $V^u_s(f)$ and $K_l$ are closed in $\mathbb{R}^l$, so is $S$. Let $$\iota_u:V^u_s(f)\cap K_l\to \mathbb{C}^d$$ be the inclusion $(x_1,...,x_l) \mapsto x_u=(x_1,...,x_1,...,x_l,...,x_l)$, where $x_i$ is repeated $u_i$ times. Then $\Pi_u(x) = (\Pi \circ \iota_u)(x)$ and clearly $\iota_u(S)$ is a closed subset of $\mathbb{C}^d$.

Let $h= t^d+h_1t^{d-1}+...+h_d\notin \Pi_u(S)$ have the roots $a=(a_1,...,a_d)$ and let $h_i=f_i \ \forall \ i\in [s]$. Let $\epsilon>0$ be such that $D_\epsilon(\sigma(a))\cap \iota_u(S)$ is empty for any $\sigma\in \Sym([d])$. If $b_1,...,b_k$ are the distinct roots of $h$ with respective multiplicities $v_1,..,v_k$, then there is a $\delta>0$ such that any polynomial, $g$, of degree $d$, with $|h_i-g_i|\leq\delta$ for all $i\in [d]$ has exactly $v_i$ zeroes in $D_\epsilon(b_i)$.

The proof of this statement can be found in \cite{uherka1977continuous}. Thus, since $D_\epsilon(\sigma(a))\cap \iota_u(S)$ is empty for any $\sigma\in \Sym([d])$, then so is $B_\delta(h)\cap \Pi_u(S)$ and therefore is $\Pi_u(S)$ closed in $\mathbb{R}^{d-s}$.
\end{proof}

\begin{proposition}\label{prop:contractible}
The sets $V^u_s(f)\cap K_l$ and $H_s^u(f)$ are contractible or empty.
\end{proposition}

\begin{proof}
Firstly, we can use Newton's identities to define $V^u_s(f)$ in terms of the first $s$ power sums in $d$ variables. Then the proof that $V^u_s(f)\cap K_l$ is contractible or empty can be found in \cite{kostov1989geometric} (Theorem 1.1). By Lemma \ref{lem:closedViete} the map $\Pi_u : V^u_s(f)\cap K_l \to H_s^u(f)$ is a homeomorphism, thus $H_s^u(f)$ is contractible if it is nonempty.
\end{proof}

To see how this proposition can be used to further describe our strata we need some more definitions. First note that as $H_s^u(f)$ is the image of a semi algebraic set under a polynomial mapping, it is semi algebraic. Thus the dimension of the stratum is the maximum integer, $n$, such that $H_s^u(f)$ contains an open set which is homeomorphic to an open set of $\mathbb{R}^n$ (see Chapter 2.8 of \cite{bochnak2013real}).

\begin{definition}\label{def:relintbd}
If $H_s^u(f)$ is a nonempty stratum and $n=\dim(H_s^u(f))$, then
\begin{itemize}
\item the \textbf{relative interior} of $H_s^u(f)$ is the set of polynomials $h\in H_s^u(f)$ such that an open neighbourhood of $h$ is homeomorphic to an open set in $\mathbb{R}^{n}$ and 
\item the \textbf{relative boundary} of $H_s^u(f)$ is the set of polynomials in $H_s^u(f)$ which is not in the relative interior.
\end{itemize}
\end{definition}

We can use Proposition \ref{prop:contractible} to give a description of the relative interior and relative boundary of our strata and also determine their dimension. But the first consequence of the proposition that we need is the following:

\begin{lemma}\label{lem:point}
If $\ell(u)\leq s$, then $H_s^u(f)$ contains at most one polynomial.
\end{lemma}

\begin{proof}
Suppose $h\in H_s^u(f)$ has the distinct roots $a=(a_1,...,a_k)$ and composition $v=(v_1,...,v_k)$, then $k\leq \ell(u)\leq s$. If $k=1$, then $v=(d)$ and there is only one solution to the equation $$-e_1(x_v)=-dx_1=f_1.$$ And so we have $H_s^v(f)=H_k^v(f) = \{h\}$. If $k>1$, then as previously mentioned, $V_s^v(f)$ can be defined as 
$$\{x\in \mathbb{R}^k| p_1(x_u) = c_1,...,p_k(x_u) = c_k\},$$
where $p_i$ is the $i^{th}$ power sum in $d$ variables and $c_1,...,c_k\in \mathbb{R}$ are obtained from the numbers $-f_1,...,(-1)^kf_k$ using Newton's identities.

The map $F:\mathbb{R}^k\to\mathbb{R}^k$, where $F(x)=(p_1(x_v),...,p_k(x_v))$, is a continuously differentiable function whose Jacobian matrix is $(iv_jx_j^{i-1})_{i,j \leq k}$ and so its determinant is $$\prod_{i=1}^k iv_i \prod_{1\leq j< r\leq k} (x_j-x_r).$$ Since all the $a_i$'s are distinct, the determinant is nonzero at $a$. Thus the Jacobian matrix is invertible and by the inverse function theorem, $F$ is invertible on some neighbourhood $U$ of $F(a) = (c_1,...,c_k)$. By Proposition \ref{prop:contractible}, $V_k^v(f)\cap K_k$ is contractible and since $a$ is isolated in this set, it must be the only point there. Therefore we have $H_s^v(f)=H_k^v(f) = \{h\}$.

So for any composition $w\leq u$, that occurs in $H_s^u(f)$, we have that $H_s^w(f)$ is a point. Since there are finitely many compositions smaller than or equal to $u$, $H_s^u(f)$ contains finitely many points. But since $H_s^u(f)$ is contractible it can contain at most one point.
\end{proof}

We will let $$P^r:\mathbb{R}^k\to\mathbb{R}^{k-r}$$ denote projection that forgets the last $r$ coordinates. This map will help us describe the relative interior.

\begin{proposition}\label{prop:derhomeom}
If $l=\ell(u)> s$, then the map $$P^{d-l}:H_s^u(f) \to \mathbb{R}^{l-s}$$ is a homeomorphism onto its image and the image is closed in $\mathbb{R}^{l-s}$.
\end{proposition}

\begin{proof}
Firstly we consider the case when $l=1$, thus $u=(d)$ and $s=0$. So for any $a\in \mathbb{R}$ we have that $(t-a)^d= t^d-dat^{d-1}+...+(-a)^d\in H_0^u(f).$ Thus $P^{d-1}((t-a)^d) = t^d-dat^{d-1}$ and so the map $$P^{d-l}\circ \Pi_u:\mathbb{R}\to \mathbb{R}$$ is essentially just mapping $a$ to $-da$. This is naturally a homeomorphism and since, by Lemma \ref{lem:closedViete}, $\Pi_u$ is a homeomorphism, then so is $P^{d-l}$. Lastly, since the image of $P^{d-l}$ is all of $\mathbb{R}$, the image is closed in $\mathbb{R}$.

Next suppose $l\geq 2$. By Lemma \ref{lem:point}, the polynomials of $H_s^u(h)$ are uniquely determined by their first $l$ coefficients, thus $P^{d-l}$ is bijection between $H_s^u(f)$ and $P^{d-l}(H_s^u(f))$. Also, the topology on $H_s^u(f)$ is the subspace topology of the product topology on $\mathbb{R}^{d-s}$ with respect to the projections on each coordinate. Thus, by definition, the map $P^{d-l}$ is continuous.

To see that the inverse is continuous and that $P^{d-l}(H_s^u(f))$ is closed we will show that the image of closed subsets of $H_s^u(f)$ are closed in $\mathbb{R}^{l-s}$. So let $S$ be a closed subset of $H_s^u(f)$, then $S$ is also a closed subset of $\mathbb{R}^{d-s}$ since, by Lemma \ref{lem:closedViete}, $H_s^u(f)$ is closed in $\mathbb{R}^{d-s}$.

Let $g$ be a point in the closure of $P^{d-l}(S)$. Then for any $\epsilon>0$, the closed ball $\overline{B}_\epsilon(g)$ meets $P^{d-l}(S)$ and so the inverse image $$(P^{d-l})^{-1}(\overline{B}_\epsilon(g)\cap P^{d-l}(S)) =(\overline{B}_\epsilon(g)\times \mathbb{R}^{d-l})\cap S$$
is nonempty. It is also closed since $\overline{B}_\epsilon(g)\times \mathbb{R}^{d-l}$ and $S$ are closed in $\mathbb{R}^{d-s}$.

Since $\Pi_u$ is a homeomorphism, then $$M=\Pi_u^{-1}((\overline{B}_\epsilon(g)\times \mathbb{R}^{d-l})\cap S)$$ is a closed subset of $V_s^u(f)\cap K_l$ and since $V_s^u(f)\cap K_l$ is closed in $\mathbb{R}^{l}$, then so is $M$. We also have that $$M\subseteq \Pi_u^{-1}(S)\cap \{x\in K_l| g_i-\epsilon \leq (-1)^ie_i(x_u)\leq g_i+\epsilon \ \forall \ i\in [l]\}.$$ So if $a\in M$, then $e_1(a_u)\leq g_1-\epsilon$ and $e_2(a_u)\geq g_2-\epsilon$ since $l\geq 2$. Thus, by Newton's identities, we have $$p_2(a_u) = e_1^2(a_u)-2e_2(a_u) \leq g_1^2-2g_1\epsilon+\epsilon^2-2g_2+2\epsilon$$ and so $M$ is bounded. Since $M$ is closed and bounded, it is compact.

Since $\Pi_u$ and $P^{d-l}$ are continuous, $P^{d-l}\circ \Pi_u$ is continuous. And since the continuous image of a compact set is compact, we have that $(P^{d-l}\circ \Pi_u)(M)$ is compact. Thus $$g\in (P^{d-l}\circ \Pi_u)(M) \subseteq P^{d-l}(S)$$ and so $P^{d-l}(S)$ is closed in $\mathbb{R}^{l-s}$. Therefore $P^{d-l}$ is a closed map and thus $P^{d-l}$ is a homeomorphism. Lastly, by setting $S=H_s^u(f)$, we see that $P^{d-l}(H_s^u(f))$ is closed in $\mathbb{R}^{l-s}$.
\end{proof}

We see from Lemma \ref{prop:derhomeom} that when $\ell(u)\geq s$, the largest dimension that $H_s^u(f)$ can have is $\ell(u)-s$. Thus we say that the \textbf{maximal dimension} of $H_s^u(f)$ is $\max\{\ell(u)-s,0\}$.

\begin{theorem}\label{thm:relint}
If $H_s^u(f)$ contains a polynomial with composition $u$, then $H_s^u(f)$ is maximal dimensional and its relative interior consists of the polynomials with composition $u$.
\end{theorem}

\begin{proof}
If $s=0$ and $l=\ell(u)$, the map $P^{d-l}\circ \Pi_u:K_l\to P^{d-l}(H_0^u(f))$ is a homeomorphism by Lemma \ref{lem:closedViete} and Lemma \ref{prop:derhomeom}. So since the dimension of $K_l$ is $l$ and its interior are the points with no repeated coordinates, then the dimension of $P^{d-l}(H_0^u(f))\subseteq \mathbb{R}^l$ is $l$ and its interior are the polynomials with composition $u$.

Next suppose $s>0$ and let $A^r_{f_i}$ denote the affine hyperplane of $\mathbb{R}^r$ defined by fixing the $i^{th}$ coordinate to be equal to $f_i$. Then $$P^{d-l}(H_s^u(f)) = P^{d-l}(H_0^u(f)\cap A^d_{f_1}\cap ... \cap A^d_{f_s})=$$
$$P^{d-l}(H_0^u(f))\cap A^l_{f_1}\cap ... \cap A^l_{f_s}.$$
Thus if there is a polynomial $h\in H_s^u(f)$ with composition $u$, then $P^{d-l}(h)$ lies in the interior of $P^{d-l}(H_0^u(f))$ and therefore also in the interior of $P^{d-l}(H_s^u(f))$ and $P^{d-l}(H_s^u(f))$ must be of dimension $\max\{l-s,0\}$. Since $P^{l-s}$ is a homeomorphism, $H_s^u(f)$ is maximal dimensional and $h$ is in its relative interior.

For the reverse inclusion, suppose $l>s$ so that $H_s^u(f)$ is at least one dimensional. If its relative interior contains a polynomial $g$ with $v(g)<u$, then $P^{d-l}(g)$ lies in the interior of the $(l-s)$-dimensional set $P^{d-l}(H_s^u(f))\subseteq\mathbb{R}^{l-s}$. Thus $P^{d-l}(g)$ lies in the interior of the one dimensional set $P^{d-l}(H_{l-1}^u(g))$.

By Lemma \ref{lem:point}, there are finitely many polynomials in $H_{l-1}^u(g)$ with a smaller composition than $u$. Thus there are two polynomials $p_-$ and $p_+$ in $H^u_{l-1}(g)$, with composition $u$, and a $\delta>0$ such that $$P^{d-l}(p_{-})=P^{d-l}(g)-\delta \text{ and } P^{d-l}(p_{+})=P^{d-l}(g)+\delta.$$

Since $P^{d-l}(p_{-})$ and $P^{d-l}(p_{+})$ are interior points of $P^{d-l}(H^u_{0}(f))$, there is an $\epsilon>0$ such that $B_{\epsilon}(P^{d-l}(p_{-}))$ and $B_{\epsilon}(P^{d-l}(p_{+}))$ are contained in the interior of $P^{d-l}(H^u_{0}(f))$. And since $P^{d-l}(g)$ is in the boundary of $P^{d-l}(H^u_{0}(f))$, the ball $B_{\epsilon}(P^{d-l}(g))$ must contain a polynomial $q = t^{d}+q_1t^{d-1}+...+q_lt^{d-l}$ that is not in $P^{d-l}(H^u_{0}(f))$.

Thus $A^l_{q_{1}}\cap ...\cap A^l_{q_{l-1}}$ is a line that passes through $q$ and the two balls $B_{\epsilon}(P^{d-l}(p_{-}))$ and $B_{\epsilon}(P^{d-l}(p_{+}))$. But if $q$ separates the nonempty sets $$B_{\epsilon}(P^{d-l}(p_{-}))\cap A^l_{q_{1}}\cap ...\cap A^l_{q_{l-1}}$$ and $$B_{\epsilon}(P^{d-l}(p_{+}))\cap A^l_{q_{1}}\cap ...\cap A^l_{q_{l-1}},$$ then $$P^{d-l}(H^u_{0}(f))\cap A^l_{q_{1}}\cap ...\cap A^l_{q_{l-1}}=P^{d-l}(H^u_{l-q1}(p_{+}))\subset \mathbb{R}$$ is nonempty but not contractible. This contradicts Proposition \ref{prop:contractible} and $g$ can therefore not be in the relative interior of $H_s^u(f)$.
\end{proof}

Thus, if $\ell(u)>s$, the stratum $H_s^u(f)$ is of maximal dimension if and only if it contains a polynomial with composition $u$. We can use this observation to determine all the possibilities for the dimension of $H_s^u(f)$. It is worth mentioning that a similar observation can be found in \cite{basu2022vandermonde}, Proposition 5.

\begin{theorem}\label{thm:consequences}
If $\ell(u)>s$ and $H_s^u(f)$ contains a polynomial with at least $s$ distinct roots, then $H_s^u(f)$ is maximal dimensional. If not, then $H_s^u(f)$ is either empty or a single polynomial.
\end{theorem}

\begin{proof}
If $s=0$, then any composition occurs and thus by Theorem \ref{def:relintbd}, any stratum is maximal dimensional. Similarly, if $s=1$, then for any composition $u$, the polynomial $-e_1(x_u)-f_1$ has a real zero with $l=\ell(u)$ distinct coordinates ordered increasingly. To see this pick $l$ real numbers $a_1,...,a_l$ such that $a_1/u_1<....<a_l/u_l$ and let
$$a=-\frac{f_1}{\sum_i a_i}\bigg{(}\frac{a_1}{u_1},...,\frac{a_l}{u_l}\bigg{)},$$
then $$-e_1(a_u) = \sum_j u_j\frac{f_1}{\sum_i a_i}\frac{a_j}{u_j} = \frac{f_1}{\sum_i a_i}\sum_ja_j = f_1.$$
Thus the composition $u$ occurs and $H_s^u(f)$ is of maximal dimension.

Next we suppose $s\geq 2$. If $\ell(u)\leq s$ or $H_s^u(f)$ does not contain a polynomial with at least $s$ distinct roots, then $H_s^u(f) = \cup_{w\leq u|\ell(w)=s-1}H_s^w(f)$. By Lemma \ref{lem:point}, $H_s^w(f)$ contains at most one point when $\ell(w)= s-1$. Since there are finitely many compositions of length $s-1$, then $H_s^u(f)$ contains finitely many polynomials and since $H_s^u(f)$ is contractible it contains at most one polynomial.

So suppose $h\in H_s^u(f)$ has $k\geq s$ distinct roots and $v(h)< u$. Let $v\leq u$ be a composition that covers $v(h)$. Then $\ell(v) = k+1$ and so by Proposition \ref{prop:derhomeom}, $H_k^v(h)$ is at most one dimensional. Since $v(h)<v$ we can write $h$ as $\prod_{i=1}^{k+1} (t-a_i)^{v_i}$ and without loss of generality we may assume that $a_1<...<a_k=a_{k+1}$.

As in the proof of Lemma \ref{lem:point}, if we define $V_k^v(f)$ as $\{x\in \mathbb{R}^{k+1}| p_1(x_u) = c_1,...,p_k(x_u) = c_k\}$, then the Jacobian matrix of the defining polynomials is $(iv_jx_j^{i-1})_{i\leq k,j\leq k+1}$ and the determinant of the leftmost $k\times k$ submatrix is $$\prod_{i=1}^k iv_i \prod_{1\leq j< r\leq k} (x_j-x_r).$$

Since the first $k$ coordinates of $a=(a_1,...,a_{k+1})$ are distinct, the determinant does not vanish at $a\in V_k^v(h)$. So by Proposition 3.3.10 in \cite{bochnak2013real}, $a$ is a nonsingular point of a one dimensional irreducible component, $V$, of $V_k^v(h)$. Thus $a$ lies in an open neighbourhood $U$ of $V$ where $U$ is a one dimensional manifold.

By Lemma \ref{lem:point}, the a one dimensional manifold, $U$, only intersects the hyperplane $H=\{x\in \mathbb{R}^{k+1}|x_k=x_{k+1}\}$ once. So $U$ must meet the open halfspace $H^{+}:=\{x\in \mathbb{R}^{k+1}|x_k<x_{k+1}\}$ and thus there is a point in $V_k^v(h)$ with no repeated coordinates. So $H_k^v(h)\subseteq H_s^u(f)$ contains a polynomial, $g$, with composition $v$.

We can apply the same argument to the polynomial $g$ if $v<u$, and keep doing this inductively until we find a polynomial with composition $u$. Then by Theorem \ref{thm:relint}, $H_s^u(f)$ is maximal dimensional.
\end{proof}

\begin{corollary}\label{cor:closure}
Any stratum equals the closure of its relative interior.
\end{corollary}

\begin{proof}
By Theorem \ref{thm:consequences} we may suppose $H_s^u(f)$ is maximal dimensional and at least one dimensional. We will prove the statement by induction on the dimension of the strata. If $H_s^u(f)$ is one dimensional, then it is connected by Proposition \ref{prop:contractible}. Thus it contains at most two relative boundary points and any open ball about a point of $H_s^u(f)$ must contain infinitely many points from the relative interior. Thus $H_s^u(f)$ is the closure of its relative interior.

Now suppose the statement is true for all $(n-1)$-dimensional strata, where $n-1\geq 1$. Suppose $H_s^u(f)$ is $n$-dimensional and $h\in H_s^u(f)$. By Proposition \ref{prop:contractible}, $H_s^u(f)$ is connected, so for any $\epsilon>0$, $B_\epsilon(h)$ must contain infinitely many polynomials from $H_s^u(f)$. Since there are finitely many polynomials in $H_s^u(f)$ with at most $s$ distinct roots, there is a $g\in B_\epsilon(h)\cap H_s^u(f)$ with at least $s+1$ distinct roots.

Then by Theorem \ref{thm:consequences}, $H_{s+1}^u(g)$ is $(n-1)$-dimensional and by the inductive hypothesis, $g$ is in the closure of its relative interior. So by Theorem \ref{thm:relint}, any open ball about $g$ contains a polynomial with composition $u$. Thus $B_\epsilon(h)$ also contains a polynomial with composition $u$ and so $h$ is in the closure of the relative interior of $H_s^u(f)$.
\end{proof}

\section{Combinatorial structure}\label{section:comb}

In this section we explore the combinatorial structure of the lattice of strata and show that this lattice is graded, atomic and coatomic. Then we use this to determine which compositions actually occur in a given starved polytope. Due to Theorem \ref{thm:consequences} we may assume $H_s(f)$ is of dimension $d-s>0$ for this section since the one-point lattice trivially satisfies the main combinatorial properties that follow. We begin with the notion of a graded poset.

\begin{definition}
A totally ordered subset of a poset is a \textbf{chain} and if a chain is maximal with respect to inclusion it is a \textbf{maximal chain}. A poset in which every maximal chain has the same length is called \textbf{graded}.
\end{definition}

To see why we call such a poset graded: let $y_0<...<y_l$ and $z_0<...<z_l$ be two maximal chains of a graded poset $L$ where $y_i=z_j$, for some $i$ and $j$. Then we have $i=j$, otherwise $y_0<y_1<...<y_i=z_j<z_{j+1}<...<z_l$ is a maximal chain which is not of length $l+1$ contradicting the gradedness of $L$. Thus the \textbf{rank} of $y_i$, $\rank(y_i):=i$, is well defined and we can write the poset as the disjoint union $L=\sqcup_{k\geq 0} L(k)$, where the elements of $L(k)$ are the rank $k$ elements.

\begin{lemma}\label{lem:comp}
If $s\geq 2$ and $H^u_s(f)$ is at least one dimensional, it is compact and its lattice of strata contains the empty stratum.
\end{lemma}

\begin{proof}
As observed in \cite{riener2012degree} Proposition 4.1, a stratum of $H_s(f)$ is compact when $s\geq 2$ since we can rewrite the $s$ first elementary symmetric polynomials in terms of the $s$ first power sums and the second power sum defines a sphere.

For the second statement, note that when $s=2$, the equations $$p_1(x) = c_1 \text{ and } p_2(x) = c_2$$ defines a nonempty intersection of a hyperplane, whose normal vector is $(1,1,...,1)$, and a $(d-1)$-sphere. This intersection can either be a $(d-2)$-sphere in the hyperplane or it can be one of the two points $\pm\sqrt{\frac{c_2}{d}}(1,1,...,1)$.

That is, if $v=(d)$, and the stratum $H_s^v(f)$ is nonempty, then there can be no other point in $H_s(f)$ since there is no other point in $H_2(f)$. So since $H_s(f)$ is at least one dimensional we have $H_s^v(f)=\emptyset$ and since $v$ is smaller than any composition $u$, then $H^u(f)$ contains the empty stratum.
\end{proof}

For $s=0$ and $s=1$, the starved polytopes will have all the main combinatorial properties we are establishing. But the argument is different from the other cases, so we will restrict $s$ to be at least $2$ for now. This allows us to use Lemma \ref{lem:comp}, which will be very helpful. Before we proceed note that we use the convention that the empty set has dimension $-1$.

\begin{proposition}\label{prop:graded}
If $s\geq 2$, the lattice of strata of $H_s(f)$ is graded and the rank of a stratum is one more than its dimension.
\end{proposition}

\begin{proof}
Suppose a stratum $H_s^u(f)$ is strictly contained in $H_s^v(f)$, then by Theorem \ref{thm:consequences}, $\dim(H_s^u(f))<\dim(H_s^v(f))$. Also, by Lemma \ref{lem:comp}, the lattice of strata contains the empty set as its minimal element. Thus any maximal chain in the lattice of strata has length at most $\dim(H_s(f))+1=d-s+1$.

Conversely, suppose we have that $H_s^u(f)$ is strictly contained in $H_s^v(f)$, where $\dim(H_s^u(f))<\dim(H_s^v(f))-1$, let us show that there is a stratum, $H_s^w(f)$, with $H_s^u(f)\subsetneq H_s^w(f)\subsetneq H_s^v(f)$.

If $H_s^u(f)$ is empty, $H_s^v(f)$ is at least one dimensional and by Theorem \ref{thm:relint} its relative interior are the polynomials with composition $v$. But by Lemma \ref{lem:comp}, it is compact and so it must contain a nonempty stratum, $H_s^w(f)$, in its relative boundary. Then by Lemma \ref{lem:comp}, both $H_s^v(f)$ and $H_s^w(f)$ contains the empty stratum $H_s^u(f)$.

If $h\in H_s^u(f)$ and there are no polynomials in $H_s^u(f)$ with at least $s$ distinct roots, then by Theorem \ref{thm:consequences}, $H_s^u(f) =\{h\}$. Since $H_s^u(f)$ is zero dimensional, $H_s^v(f)$ is at least two dimensional and by Lemma \ref{lem:comp} and Proposition \ref{prop:contractible}, it is compact and contractible. Thus its relative boundary is at least one dimensional and connected.

If $H_s^u(f)$ is not contained in a larger stratum of $H_s^v(f)$, it must be an isolated part of the relative boundary of $H_s^v(f)$. But since the relative boundary is one dimensional and connected, $H_s^u(f)$ must be the whole boundary. But this is impossible since $H_s^u(f)$ is zero dimensional, thus there is a stratum $H_s^w(f)$ which is strictly contained in $H_s^v(f)$ and that strictly contains $H_s^u(f)$.

Lastly, if there is a $h\in H_s^u(f)$ with at least $s$ distinct roots, then we may assume $u=v(h)$ by Theorem \ref{thm:consequences}. Also, we then have that all compositions greater than $u$ occurs in $H_s(f)$. Since $$\ell(u)-s=\dim(H_s^u(f))<\dim(H_s^v(f))-1=\ell(v)-s-1,$$ then $\ell(u)<\ell(v)-1$ and so there is a composition $w$, with $u<w<v$. By Theorem \ref{thm:consequences}, $H_s^w(f)$ is of dimension $(\ell(w)-s)$ and therefore $H_s^u(f)\subsetneq H_s^w(f)\subsetneq H_s^v(f)$.

Thus any maximal chain will be at least of length $d-s+1$ and so any maximal chain has length $d-s+1$. Also, by the above argument any stratum of dimension $n\geq 0$ covers a stratum of dimension $n-1$, thus its rank must be $n+1$.
\end{proof}

Next up is the notion of atomic lattices.

\begin{definition}
In a lattice with a smallest element, $0$, the elements covering $0$ are called \textbf{atoms}. The lattice is called \textbf{atomic} if any element can be expressed as the join of atoms.
\end{definition}

\begin{lemma}\label{lem:contains}
If $s\geq 2$, then any $n$-dimensional stratum, where $n>0$, contains at least two distinct $(n-1)$-dimensional strata.
\end{lemma}

\begin{proof}
By Proposition \ref{prop:graded}, an $n$-dimensional stratum, $H_s^u(f)$, contains an $(n-1)$-dimensional stratum $H_s^v(f)$. Also, by Lemma \ref{lem:comp}, a nonempty stratum $H_s^u(f)$ is compact and so its relative boundary is nonempty but not contractible. But by Proposition \ref{prop:contractible}, $H_s^v(f)$ is contractible and so it cannot be the whole relative boundary of $H_s^u(f)$.

Also, since $H_s^v(f)$ is closed, then $\Relbd(H_s^u(f))\backslash H_s^v(f)$ is relatively open in $\Relbd(H_s^u(f))$ with the subspace topology. Thus $\Relbd(H_s^u(f))\backslash H_s^v(f)$ is $(n-1)$-dimensional and so there must be another $(n-1)$-dimensional stratum in $H_s^u(f)$.
\end{proof}

\begin{proposition}\label{prop:atomic}
If $s\geq 2$, the lattice of strata of $H_s(f)$ is atomic.
\end{proposition}

\begin{proof}
By convention the empty set is the join of an empty set of atoms and an atom is naturally the join of itself. Also, by Proposition \ref{prop:graded}, the lattice is graded and a stratum's rank is its dimension plus one, so the atoms are the zero dimensional strata.

If $H_s^u(f)$ is an $n$-dimensional stratum, where $n>0$, then by Lemma \ref{lem:contains}, there are two distinct $(n-1)$-dimensional strata, $H_s^v(f)$ and $H_s^w(f)$, contained in $H_s^u(f)$. Since $H_s^v(f)$ and $H_s^w(f)$ are distinct, by Proposition \ref{prop:graded}, any stratum that contains both must be at least $n$-dimensional. Since $H_s^u(f)$ is $n$-dimensional and contains both $H_s^v(f)$ and $H_s^w(f)$, it must be the join of $H_s^v(f)$ and $H_s^w(f)$. By induction, both $H_s^v(f)$ and $H_s^w(f)$ are joins of atoms and since $H_s^u(f)$ is the join of $H_s^v(f)$ and $H_s^w(f)$, it must also be a join of atoms.
\end{proof}

Lastly, we look at a sort of converse of atomic lattices called coatomic lattices.

\begin{definition}
In a lattice with a largest element, $1$, the elements covered by $1$ are called \textbf{coatoms}. The lattice is called \textbf{coatomic} if any element can be expressed as the meet of coatoms.
\end{definition}

\begin{lemma}\label{lem:contin}
If $s\geq 2$, then any $n$-dimensional stratum, where $n<d-s-1$, is contained in at least two distinct $(n+1)$-dimensional strata.
\end{lemma}

\begin{proof}
Let $H_s^u(f)$ be an $n$-dimensional stratum. If $n>0$, by Theorem \ref{thm:consequences}, it must be maximal dimensional. Thus there is a polynomial with composition $u$ and since $\ell(u)-s=n<d-s-1$, then $\ell(u)<d-1$, so $u$ is covered by at least two distinct compositions, $v$ and $w$, of length $\ell(u)+1$. So again by Theorem \ref{thm:consequences}, $H^v_s(f)$ and $H^w_s(f)$ must be $(n+1)$-dimensional.

If $n=0$, then $H_s^u(f)=\{h\}$ and since $d-s>1$, the set $H_s(f)$ is at least two dimensional. By Proposition \ref{prop:graded}, $h$ is contained in a two dimensional stratum, $H_s^w(f)$, and a one dimensional stratum, $H_s^v(f)\subset H_s^w(f)$. By Theorem \ref{thm:relint}, $H_s^v(f)$ is in the one dimensional relative boundary of $H_s^w(f)$ and $h$ is in the relative boundary of $H_s^v(f)$.

According to Corollary \ref{cor:closure}, $H_s^w(f)$ is the closure of its relative interior, so there is a connected component, $C$, of $\Relint(H_s^w(f))$ such that $h$ is in its closure, $\overline{C}$. Thus, starting from $h$, we can traverse its boundary clockwise or counter-clockwise. But since $h$ is one of the relative  boundary points of $H_s^v(f)$, at most one of the directions consists immediately of polynomials whose composition is $v$. Thus there must be some other one dimensional stratum for which $h$ is a boundary point.

Lastly, if $n=-1$, then $H_s^u(f)$ is empty. Since $H_s(f)$ is at least one dimensional and, by Proposition \ref{prop:atomic}, the lattice of strata is atomic, it must contain at least two atoms. Thus the empty stratum is contained in at least two zero dimensional strata.
\end{proof}

\begin{proposition}
If $s\geq 2$, the lattice of strata of $H_s(f)$ is coatomic.
\end{proposition}

\begin{proof}
The argument is analogous to the proof of Proposition \ref{prop:atomic}, just start the induction from the $(d-s-1)$-dimensional strata and use Lemma \ref{lem:contin} instead of Lemma \ref{lem:contains} for the induction step.
\end{proof}

\begin{theorem}\label{thm:latticeprops}
The lattice of strata of $H_s(f)$ is graded, atomic and coatomic.
\end{theorem}

\begin{proof}
The only cases that remains to check is when $s\leq 1$. As we saw in the proof of Theorem \ref{thm:consequences}, all compositions occur when $s\leq 1$ so the lattice is isomorphic to the lattice of compositions. The lattice of compositions is isomorphic to the face lattice of a $(d-2)$-dimensional simplex, $S$. To see this, note that any $k+1$ vertices in $S$ determines a $k$-dimensional face of $S$.

Similarly there are $d-1$ compositions of $d$ of length $2$. So if $v^{i_1},...,v^{i_{k+1}}$ are $k+1$ distinct compositions of length $2$, then their first parts are distinct. We may assume they are ordered so that $v^{i_1}_1<...<v^{i_{k+1}}_1$ and so, by the argument in Lemma \ref{lem:latticecomps}, their join is $$v^{i_1}\vee ... \vee v^{i_{k+1}} = (v^{i_1}_1,v^{i_2}_1-v^{i_1}_1,v^{i_3}_1-v^{i_2}_1,...,v^{i_{k+1}}_1-v^{i_k}_1,d-v^{i_{k+1}}_1),$$
which is a composition of length $k+2$. Thus any bijection from the set of length $2$ compositions to the set of vertices of $S$ induces an isomorphism of lattices. Lastly, by point (i) and (v) of Theorem 2.7 in \cite{ziegler2012lectures}, the face lattice of a simplex is graded, atomic and coatomic.
\end{proof}

Similar to the cases when $s\leq 1$, when $s=2$ it can be shown that the lattice of strata is isomorphic to the face lattice of a simplex. Although this does no need to be true in general, as can be seen by intersecting $H_2(f)$, from Example \ref{ex:tetra}, with the affine hyperplane defined by fixing the third coefficient to be $0$. This gives us a lattice which is isomorphic to the face lattice of a quadrilateral.

Thus there are examples for which the lattice of strata is not isomorphic to the face lattice of a simplex, however it would be interesting to find an answer to the following question for general $d$ and $s$.

\begin{question}
Is the lattice of strata of a starved polytope polytopal?
\end{question}

We finish this section with an algorithm to compute which compositions occur in $H_s(f)$ based on the compositions of length at most $s$ that occurs.

\begin{algorithm}\label{alg:comps}
Suppose $H_s(f)$ is $(d-s)$-dimensional and that $s\geq 2$. Let $U$ denote the set of compositions in $H_s(f)$ of length at most $s$.

Step 1: Compute the join of every pair of compositions in $U$: $$V:=\{u\vee v|  u,v \in U\ \& \ u\neq v\}.$$

Step 2: Compute the upward closure of $V$: $$W:=\{w|\exists \ v\in V \text{ with } v\leq w \}.$$

Then $U\cup W$ is the set of all compositions occurring in $H_s(f)$.
\end{algorithm}

\begin{proof}
Let $w\in W$, then $w\geq v\vee u$ for some $v,u\in U$. Thus both $v$ and $u$ occur in $H_s(f)$ and so $H_s^w(f)$ contain at least two polynomials. So by Theorem \ref{thm:consequences}, $H_s^w(f)$ is maximal dimensional and therefore there is a polynomial with composition $w$. Thus all compositions computed in the algorithm occurs in $H_s(f)$.

Suppose a composition $u$, with $\ell(u)>s$, occurs in $H_s(f)$. Then by Theorem \ref{thm:consequences}, $H_s^u(f)$ is at least one dimensional and by Theorem \ref{thm:latticeprops}, $H_s^u(f)$ is the join of at least two distinct atoms $H_s^v(f)=\{h\}$ and $H_s^w(f)=\{g\}$. We may assume $v$ and $w$ are the compositions of $h$ and $g$ respectively. Then $v\vee w \in V$ and $v\vee w \leq u$, thus $u\in W$ and it was not left out by the algorithm.
\end{proof}

\begin{remark}
Step 1 in Algorithm \ref{alg:comps} can be accomplished using the method described in Lemma \ref{lem:latticecomps}. That is, the join of $u$ and $v$ can be computed by first constructing the set $$M=\{u_1,u_1+u_2,...,d,v_1,v_1+v_2,...,d\}.$$ Next, construct the tuple $(m_1,....,m_k)$, where the $m_i$'s are distinct, increasingly ordered and $\{m_1,...,m_k\}=M$. Then the join of $u$ and $v$ is the composition $$(m_1,m_2-m_1,m_3-m_2,...,m_l-m_{l-1}).$$ However, compositions and our partial order are both implemented in Sage (see \cite{sagemath}), so the algorithm can easily be implemented there.
\end{remark}

\begin{remark}
As in \cite{arnol1986hyperbolic}, \cite{givental1987moments} and \cite{kostov1989geometric} we could have focused on the intersection of Vandermonde varieties and $K$. That is, we could consider the set
$$M=\{x\in \mathbb{R}^d|w_1x_1^i+w_2x_2^i+...+w_dx_d^i=c_i \ \forall \ i\in [s]\}\cap K$$
were $w_1,...,w_d\in \mathbb{R}$ are positive and $c_1,...,c_s\in\mathbb{R}$. If $x\in M$ is of the form $x_1=...=x_{v_1}<x_{v_1+1}=...=x_{v_1+v_2}<...<x_{d-v_l+1}=...=x_d$, we associate to it the composition $v(x)=(v_1,v_2,...,v_d)$ and for a composition $u$ we may define a stratum of $M$ as $M^u=\{y\in M|v(y)\leq u\}$.

When the weights $w_1,...,w_d$ are integers (and by extension rational) then $M^u$ is equal to $\iota_u (V_s^u(f)\cap K_l)$ (see the proof of Lemma \ref{lem:closedViete}) for some monic, univariate hyperbolic polynomial $f$ of degree $d$. However, if the weighs are irrational we do not see how to interpret the set $M$ and so we did not consider such cases. But it can be shown that for any positive real weights Theorem \ref{thm:relint} and Theorem \ref{thm:consequences} also hold for $M^u$ and that Theorem \ref{thm:latticeprops} hold for the poset of strata of $M$.

As we lack the motivation to consider such cases we do not go into the details. But the main difference in the arguments would be to use the map $W^l:M\to \mathbb{R}^l$ given by $$x\mapsto \bigg{(}\sum_{j=1}^dw_jx_j,...,\sum_{j=1}^dw_jx_j^l\bigg{)},$$ instead of using the map $P^{d-l}\circ \Pi_u$ in the proof of Theorem \ref{thm:relint}, otherwise the arguments follow through the same way.
\end{remark}

\section{A short note on starvation}\label{starve}

We have seen that the lattice of strata of a starved polytope has many properties similar to the face lattices of polytopes, so now we discuss some results on the boundary of starved polytopes that explains the description "starved". To do this we will quickly introduce discriminants and subdiscriminants, but more information on these objects can be found in \cite{gelfanddiscriminants} and in Chapter 4 of \cite{basualgorithms}.

Let $\Delta_d$ denote the \textbf{discriminant} of a real, monic, univariate polynomial of degree $d$. This is a real polynomial in the coefficients of the univariate polynomial which has degree $d$ and it vanishes when the corresponding univariate polynomial has a repeated root.

Let $Z(\Delta_d)\subseteq \mathbb{R}^d$ denote the real algebraic set given by the discriminant. Then the points of $Z(\Delta_d)$ that corresponds to polynomials with a repeated real root, splits the space of real coefficients into $\lfloor d/2\rfloor$ regions, each of which is characterised by the number of real roots that the polynomials have.

Similarly, we let $\Delta_{d,k}$ denote the $k^{th}$ \textbf{subdiscriminant} of a real, monic, univariate polynomial of degree $d$. This is also a real polynomial in the coefficients of the univariate polynomial and the $k$ first subdiscriminants vanish when the corresponding univariate polynomial has at most $d-k$ distinct roots.

So if $h=t^d+h_1t^{d-1}+...+h_d$ lies in the real algebraic set given by the first $k$ subdiscriminants, $Z(\Delta_{d,0},...,\Delta_{d,k})\subseteq\mathbb{R}^d$, then $h$ has at most $d-k$ distinct roots. If $h$ has exactly $d-k$ distinct roots then it was show in Proposition 1.3.4 of \cite{meguerditchian1991geometrie}, that the tangent space of $Z(\Delta_{d,0},...,\Delta_{d,k})$ at $h$, lies, in a neighbourhood of $h$, in the region which locally about $h$ has the maximal number of real roots.

That is, let $T_h$ be the tangent space of $Z(\Delta_{d,0},...,\Delta_{d,k})$ at $h$, where we consider the tangent space as having been translated to $h$ and not centred at the origin. Then there is a neighbourhood $N\subset \mathbb{R}^d$, of $h$, such that $T_h\cap N$ consists of polynomials that have exactly $n$ real roots and no other polynomial in $N$ have more than $n$ real roots.

Naturally the strata of $H_s(f)$ labelled by compositions of length $d-k$, are subsets of $Z(\Delta_{d,0},...,\Delta_{d,k})$. Now suppose $h$ is hyperbolic and has composition $u$. Then if the tangent space of $H_s^u(h)$ at $h$ is well defined, it must be included in $T_h\cap A_{h_1}\cap ... \cap A_{h_s}$, where $A_{h_i}\subset\mathbb{R}^d$ is the affine hyperplane defined by fixing the $i^{th}$ coordinate to be equal to $h_i$.

The maximal number of real roots of a polynomial in $N$ is $d$ since $h$ is hyperbolic, thus $T_h\cap N \cap A_{h_1}\cap ... \cap A_{h_s}\subset H_s(h)$ and so the tangent space of $H_s^u(h)$ at $h$, intersected with $N$, must also lie in $H_s(h)$. Thus the strata of $H_s(h)$ are, in a sense, concave towards $H_s(h)$ and so in combination with Theorem \ref{thm:latticeprops} the description "starved polytope" is fairly well justified.

\textit{Disclaimer.} No polytopes were starved in the making of this article and the author does not condone the starving of polytopes or any other mathematical objects.

\section*{Index of notation}

\begin{itemize}
\item[\textbf{-}] $[s]=\{1,2,....,s\}$
\item[\textbf{-}] $f = t^d+f_1t^{d-1}+...+f_d$
\item[\textbf{-}] $H_s(f) = \{h = t^d+h_1t^{d-1}+...+h_d|h \text{ is hyperbolic and }h_i =f_i \ \forall \ i\in [s]\}$
\item[\textbf{-}] $v(h)$ denotes the composition of $h$
\item[\textbf{-}] $H_s^u(f) = \{h\in H_s(f)|v(h) \leq u\}$
\item[\textbf{-}] $\ell(u)$ denotes the length of the composition $u$
\item[\textbf{-}] For $x\in \mathbb{R}^{\ell(u)}$, $x_u = (x_1,...,x_1,...,x_{\ell(u)},...,x_{\ell(u)})$, where $x_i$ is repeated $u_i$ times.
\item[\textbf{-}] $e_i(x)$ is the $i^{th}$ elementary symmetric polynomial in $d$ variables.
\item[\textbf{-}] $p_i(x)$ is the $i^{th}$ power sum in $d$ variables.
\item[\textbf{-}] $V^u_s(f)=\{x\in\mathbb{R}^l|-e_1(x_u)-f_1=0, ..., (-1)^se_s(x_u)-f_s=0\}$
\item[\textbf{-}] $K_l=\{x\in \mathbb{R}^l|x_1\leq ... \leq x_l\}$
\item[\textbf{-}] For $x\in \mathbb{R}^l$, $\Pi_u(x) = t^d -e_1(x_u)t^{d-1}+...+(-1)^de_d(x_u)$
\item[\textbf{-}] $B_\epsilon(a)$ is the open ball around $a\in\mathbb{R}^n$ of radius $\epsilon$.
\item[\textbf{-}] $u\vee v$ denotes the join and $u\wedge v$ denotes the meet of two elements, $u$ and $v$, of a lattice.
\end{itemize}

\nocite{riener2022linear}
\nocite{meguerditchian1992theorem}
\nocite{dedieu1992obreschkoff}
\nocite{chevallier1986courbes}
\nocite{kostov2011topics}
\nocite{stanley2011enumerative}

\newpage
\printbibliography

\Addresses

\end{document}